\documentclass[12pt]{article}

\usepackage{amsmath,amsfonts,amsthm,amssymb,latexsym}
\usepackage{esint}

\DeclareMathOperator{\supp}{supp}

\let\Re\relax\DeclareMathOperator{\Re}{Re} 

\newcommand{\ds}{\displaystyle}

\newtheorem{theorem}{Theorem}[section]
\newtheorem{lemma}[theorem]{Lemma}

\newtheorem{proposition}[theorem]{Proposition}

\theoremstyle{definition}

\theoremstyle{remark}
\newtheorem{remark}[theorem]{Remark}

\numberwithin{equation}{section}

\newcommand{\Addresses}{{
		\bigskip
		\footnotesize
		
		\noindent\textbf{Arno B.J. Kuijlaars}, \textsc{Department of Mathematics, Katholieke Universiteit Leuven, Celestijnenlaan 200B box 2400, 3001 Leuven, Belgium.}\par\nopagebreak
		\noindent\textit{E-mail address}: \texttt{arno.kuijlaars@kuleuven.be}
				
}}

\begin{document}
\title{Extremal polynomials on the $n$-grid} 
\author{Arno B.J. Kuijlaars}
\date{\small Journal of Approximation Theory 288 (2023), 105875} 
\maketitle

\begin{abstract}
The $n$-grid $E_n$ consists of $n$ equally spaced points in $[-1,1]$ including the endpoints $\pm 1$.
The extremal polynomial $p_n^*$  is the polynomial  
that maximizes the uniform norm $\| p \|_{[-1,1]}$
among polynomials $p$ of degree $\leq \alpha n$ that are bounded by one in absolute value on $E_n$.	
For every $\alpha \in (0,1)$, we determine the limit of $\frac{1}{n} \log \| p_n^*\|_{[-1,1]}$ 
as $n \to \infty$. 
The interest in this limit comes from a connection with an impossibility theorem on
stable approximation on the $n$-grid.

\end{abstract}
\section{Statement of result}
The $n$-grid from the title refers to $n$ equally spaced points in $[-1,1]$
\begin{equation} \label{En} 
	E_n = \{ \xi_{k,n} = \tfrac{2k-n-1}{n-1} \mid k = 1, \ldots, n \},
\end{equation}  
ranging from $-1$ to $1$. Let $0 < \alpha < 1$.
This paper is about the determination of the limit of the expression
\begin{equation} \label{ratio} 
	\frac{1}{n} \log \sup_{\deg p \leq \alpha n} \frac{\| p \|_{[-1,1]}}{\| p\|_{E_n}} \end{equation}
as $n \to \infty$, where the norms are uniform norms over the indicated
sets, and the supremum is over univariate polynomials $p$
of degrees at most $\alpha n$ that do not vanish identically. 
The  result was already announced in the paper \cite{PTK11} from 2011. 
Renewed interest in it is due to  \cite{HT22}. 

\begin{theorem} \label{theorem11} For every $\alpha \in (0,1)$ the limit
\begin{equation} \label{limitresult}
    \lim_{n \to \infty} \frac{1}{n}  \log \sup_{\deg p \leq \alpha n} \frac{\| p \|_{[-1,1]}}{\| p\|_{E_n}} = C(\alpha) \end{equation}
exists and is equal to
\begin{equation} \label{Calpha}  
	C(\alpha)  = \frac{(1+\alpha) \log(1+\alpha) + (1-\alpha) \log(1-\alpha)}{2}.
    \end{equation}
\end{theorem}

The limit $C(\alpha)$ is positive and strictly increasing with $\alpha$.
There is a nice Taylor expansion
\[  C(\alpha)  = \sum_{k=1}^{\infty} \frac{\alpha^{2k}}{2k(2k-1)}.  \] 
The odd Taylor coefficients vanish since the power series defines an odd function. 
The even Taylor coefficients are positive and therefore $C(\alpha) \geq \frac{1}{2} \alpha^2$. 
It is also worth noticing that $C(\alpha) \to \log 2$ as $\alpha \to 1-$.
For $\alpha = 1$ however, we have that the limit in \eqref{limitresult}
is $+\infty$, since for each $n$, there is a non-zero polynomial of degree $n$
that vanishes on the $n$-grid. 

\paragraph{Discussion}
The limit \eqref{limitresult} shows that a polynomial of degree $\leq \alpha n$
that is bounded on $E_n$ can be exponentially large somewhere in the interval $[-1,1]$. 
Namely, if $|p(\xi_{k,n})| \leq 1$ for each $k=1, \ldots, n$, then $|p(x)|$
at some $x \in [-1,1]$ can be as large as $e^{n (C(\alpha) +o(1))}$ as
$n \to \infty$,
and the constant $C(\alpha)$ is sharp.
The result is related to earlier work of
Coppersmith and Rivlin \cite{CR92} who showed that there
exist universal constants $C_2 > C_1 > 1$ such that
for $n$ large enough, and for every $d \leq n-1$,
\begin{equation} \label{CRresult}
	C_1^{d^2/n} \leq \sup_{\deg p \leq d} \frac{\| p \|_{[-1,1]}}{\| p\|_{E_n}} 
	\leq C_2^{d^2/n}. \end{equation}
The inequalities \eqref{CRresult} show that polynomials of degree $d \leq c \sqrt{n}$ that are bounded by one on $E_n$ are uniformly bounded on $[-1,1]$
with a constant that only depends on $c$.
However, if $d$ grows proportionally with $n$ then \eqref{CRresult}
shows that polynomials
that are bounded by one on $E_n$ may be exponentially large on $[-1,1]$, 
and this behavior is made more precise in
the limit \eqref{limitresult}.

The comparisons of the two uniform norms $\| \cdot \|_{E_n}$ and $\| \cdot \|_{[-1,1]}$ arises naturally when studying approximation or
interpolation methods for analytic functions based on function values
on the $n$-grid. There is a trade-off between convergence
and stability properties that was made precise in the
impossibility theorem of \cite[Theorem 3.1]{PTK11}. For example, exponential
convergence as $n \to \infty$ comes together with exponential instability.
The proof of the impossibility theorem in \cite{PTK11} relies on
the Coppersmith-Rivlin inequalities \eqref{CRresult}.
For recent work in this direction we refer to \cite{AS23, HT22}.

It was observed in \cite[section 4]{PTK11} that the phenomenon that
polynomials of degree $d \approx \alpha n$ 
can be much larger on $[-1,1]$ than on $E_n$ may be understood
in terms of potential theory. Here one thinks of a polynomial in terms
of its zeros. A polynomial may be small at certain gridpoints in $E_n$
by simply having a zero very close to these gridpoints.  
Since there are more gridpoints than zeros this cannot 
happen for every gridpoint. The extremal polynomial
$p_n^*$ for \eqref{ratio} will place a certain fraction of its zeros 
extremely close to gridpoints lying in a subset $S$ of $[-1,1]$ (with $S$ depending on $\alpha$).
Then $p_n^*$ is small at the gridpoints in $S$  but not necessarily in between,
and in fact it has high oscillations in $S$.  Following \cite{BKMM07, DS98} we call
$S$ the saturated region. 
The non-saturated region $[-1,1] \setminus S$ has   
fewer zeros than gridpoints. The extremal polynomial $p_n^*$ is 
not only small at the gridpoints in $[-1,1] \setminus S$ but 
it is of comparable size  over the full non-saturated region, 
see \cite{Rak07} for very precise estimates.

This phenomenon was first described by Rakhmanov \cite{Rak96} for orthogonal
polynomials on the $n$-grid, or more generally, for  polynomials that
minimize a discrete $L_p$ norm on $E_n$. These polynomials
have their zeros in $[-1,1]$ and they are separated by the gridpoints,
in the sense that in between any two distinct zeros there is at least
one gridpoint. In the limit $n \to \infty$ the zeros of the extremal polynomials
of degree $\lfloor \alpha n \rfloor$ considered in \cite{Rak96}
have a limiting distribution $\mu_{\alpha}$ (depending only on $\alpha$)
that is characterized by a constrained equilibrium problem from potential
theory. The measure $\mu_{\alpha}$ has a density with respect to
Lebesgue measure on $[-1,1]$ with $\frac{d\mu_{\alpha}}{dx} \leq \frac{1}{2}$
where $\frac{1}{2}$ is the limiting density of the gridpoints as $n \to \infty$. 
The saturated region $S$ is where the 
equality $\frac{d\mu_{\alpha}}{dx} = \frac{1}{2}$. holds. 
We give details in section \ref{section22} below.
The extremal polynomials $p_n^*$ for \eqref{ratio}
have the same limiting zero distribution $\mu_{\alpha}$ as $n \to \infty$,
as we will show in this paper.
Also in other aspects they behave similarly to the orthogonal polynomials
on the $n$-grid, and this will be the clue to the proof of Theorem \ref{theorem11}. 

\paragraph{Outline of the proof}
The extremal polynomial  for \eqref{ratio} is a polynomial $p_n^*$ 
of degree $\leq \alpha n$ such that 
\begin{equation} \label{pnstar}
	\frac{\| p_n^* \|_{[-1,1]}}{\| p_n^* \|_{E_n}}  = \sup_{\deg p \leq \alpha n} \frac{\| p \|_{[-1,1]}}{\| p\|_{E_n}}. \end{equation}
The proof of \eqref{limitresult} then naturally comes in two steps.
In the first step we prove the lower bound
\begin{equation} \label{lowerbound}
	\liminf_{n \to \infty} \frac{1}{n} \log \frac{\| p_n^* \|_{[-1,1]}}{\| p_n^* \|_{E_n}} 
	\geq C(\alpha) \end{equation}
and in the second step the corresponding upper bound
\begin{equation} \label{upperbound}
	\limsup_{n \to \infty} \frac{1}{n} \log \frac{\| p_n^* \|_{[-1,1]}}{\| p_n^* \|_{E_n}} 
	\leq C(\alpha). \end{equation}
The lower bound comes from considering  the $L_{\infty}$-extremal
polynomials $P_n^*$  on $E_n$, where $P_n^*$ is the monic polynomial 
of degree $\lfloor \alpha n \rfloor$ that minimizes the uniform norm $\| \cdot \|_{E_n}$.
Using the results from \cite{DS98, Rak96} and some additional calculations 
we prove in section \ref{section2} that
\begin{equation} \label{lowerbound2} 
	\lim_{n \to \infty} \frac{1}{n} \log \frac{\| P_n^*\|_{[-1,1]}}{\| P_n^*\|_{E_n}}
	= C(\alpha), \end{equation}
and this implies the lower bound \eqref{lowerbound}.

The upper bound \eqref{upperbound} is proved in section \ref{section3}.
It comes from a study of the zeros of the extremal
polynomials $p_n^*$ satisfying \eqref{pnstar}. We show in Lemma \ref{lemma31}
that the zeros are real and simple and at least $\lfloor \alpha n \rfloor -1$
zeros are in $[-1,1]$ where they are separated by the gridpoints. Note that
one zero could be in $\mathbb R \setminus [-1,1]$.
We then use potential theoretic arguments to show that the limiting
distribution of the zeros of $p_n^*$ is equal to the constrained
equilibrium measure $\mu_{\alpha}$. Along the way we prove that
$\mu_{\alpha}$ is the maximizer of a functional $J$ that we define in
\eqref{Jmudef} with $J(\mu_{\alpha}) = C(\alpha)$, which leads to
the upper bound \eqref{upperbound}.

We finally note that discrete orthogonal polynomials and the constrained
equilibrium problem also play a role in the analysis of iterative methods
from numerical linear algebra \cite{BK01,Kui00,Kui06},   and the asymptotic 
analysis of integrable systems \cite{DM98}, random matrices and 
random tiling models \cite{BKMM07, BL14}.

\section{Proof of the lower bound} \label{section2}

\subsection{Extremal polynomials on the $n$-grid}
As explained above, we are going to consider the monic polynomial
$P_n^*$ of degree $\lfloor \alpha n \rfloor$ such that
\begin{equation} \label{Pnextremal} 
	\| P_n^*\|_{E_n} = \min_{\scriptsize \begin{array}{c} \deg P = \lfloor \alpha n \rfloor \\
			P \text{ is monic}\end{array}} \| P\|_{E_n}. \end{equation}
We are going to show that  the limit
\eqref{lowerbound2} holds.

Rakhmanov \cite{Rak96} considered polynomials $P$ of degree $n$ that are monic (leading coefficient
equal to $1$) and that minimize either the uniform norm $\| P \|_{E_N}$,
or a discrete $p$-norm on $E_N$ among all such polynomials. The interest is in their asymptotic 
behavior as both $n, N \to \infty$ with $n/N \to c < 1$. The equispaced $n$-grid \eqref{En} is 
actually only a special case of far more general discrete sets that are considered in \cite{Rak96}.
Compared to \cite{Rak96} we change $N \mapsto n$, $n \mapsto \lfloor \alpha n \rfloor$, $c \mapsto \alpha$.

\subsection{Limiting behavior of zeros} \label{section22}

The zeros of $P_n^*$ are real and simple. They belong to the interval $(-1,1)$,
where they are separated by the nodes in $E_n$, see \cite{DS98,Rak96}.
To $P_n^*$ we associate
the normalized zero counting measure
\[ \nu_n = \frac{1}{n} \sum_{k=1}^{\lfloor \alpha n \rfloor}
	\delta_{x_{k,n}}, \]
where $x_{k,n}$ for $k=1, \ldots, \lfloor \alpha n \rfloor$, 
denote the zeros of $P_n^*$. Note that we normalize with the factor $1/n$
while the degree of $P_n^*$ is $\lfloor \alpha n \rfloor$. Thus 
$\nu_n$ is not a probability measure but rather has
a total mass $\frac{\lfloor \alpha n \rfloor }{n}$

Rakhmanov \cite[Theorem 2]{Rak96}, see also \cite[Theorem 3.3]{DS98}, 
proved that the weak$^*$ limit
\begin{equation} \label{nuPnlimit} 
		\nu_n \stackrel{\ast}{\to} \mu_{\alpha} \end{equation}
exists, where $\mu_{\alpha}$ is the measure on $[-1,1]$ with
density  \cite[Theorem 1]{Rak96}
\begin{equation} \label{mualpha} \frac{d\mu_{\alpha}}{dx}
	= \begin{cases} \ds \frac{1}{2}, & \text{ for } 
		x \in [-1,-r] \cup [r, 1], \\[10pt]
		\ds	\frac{1}{\pi} \arctan \left( \frac{\alpha}{\sqrt{r^2-x^2}}\right), 
		& \text{ for } x \in [-r,r], \end{cases} \end{equation}
where
\begin{equation} \label{ralpha} 
	r = r(\alpha) =  \sqrt{1-\alpha^2}. \end{equation}
The density on $[-r,r]$
can alternatively be written as
\begin{equation} \label{mualpha2}
	\frac{d\mu_\alpha}{dx}
	= 	\frac{1}{2} - \frac{1}{\pi} \arccos \left(\frac{\alpha}{\sqrt{1-x^2}}\right), 
	\quad \text{ for } x \in [-r,r].
\end{equation}

The measure $\mu_{\alpha}$ belongs to the class
\begin{equation} \label{Malphasigma} 
	\mathcal M_{\alpha, \sigma} :=
	\{ \mu \mid \smallint d\mu = \alpha, \, 0 \leq \mu \leq \sigma \} \end{equation}
where 
\begin{equation} \label{sigma}
	d \sigma = \frac{1}{2} \chi_{[-1,1]}(x) dx
\end{equation} denotes the 
Lebesgue measure restricted to $[-1,1]$ with normalization such that
$\int d\sigma =1$. The upper constraint
$\mu_{\alpha} \leq \sigma$ comes from the fact that the zeros of $P_n^*$ are separated by the nodes $\xi_{k,n}$ in the equispaced grid $E_n$.  

Rakhmanov also characterized $\mu_{\alpha}$ in terms of notions from 
logarithmic potential theory \cite{ST97}. Let 
\begin{equation} \label{Imu} 
	I(\mu) = \iint \log \frac{1}{|x-y|} d\mu(x) d\mu(y) \end{equation}
and
\begin{equation} \label{Umu}
	U^{\mu}(x) = \int \log \frac{1}{|x-y|} d\mu(y) \end{equation}
denote the logarithmic energy and the logarithmic potential of a 
measure $\mu$, respectively.  Then 
\begin{equation} \label{Imualpha}
	I(\mu_{\alpha}) = \min_{\mu \in \mathcal M_{\alpha, \sigma}} I(\mu) 
\end{equation}
and $\mu_{\alpha}$ is the unique minimizer within the class \eqref{Malphasigma}.
Furthermore, there is a constant $\ell_{\alpha}$ such that
\begin{equation} \label{Umualpha}
	U^{\mu_{\alpha}}(x) 
	\begin{cases} = \ell_{\alpha}, & \text{ for } x \in \supp(\sigma-\mu_{\alpha}), \\
		\leq \ell_{\alpha}, & \text{ for } x \in [-1,1], 
	\end{cases}
\end{equation}
and $\mu_{\alpha}$ is the only measure $\mu$ in 
$\mathcal M_{\alpha, \sigma}$ such that $U^{\mu}(x) = \ell$ is constant
on $\supp(\sigma - \mu)$ and $U^{\mu}(x) \leq \ell$ on $[-1,1]$
for a certain constant $\ell$.

Because of the upper constraint $\mu \leq \sigma$, the
measure $\mu_{\alpha}$ is called a constrained equilibrium
measure, see \cite{BKMM07,DS98}. The saturated region $S$ is 
where $ \frac{d\mu_{\alpha}}{dx} = \frac{d\sigma}{dx} = \frac{1}{2}$
and in view of \eqref{mualpha} we have $S = [-1,-r] \cup [r,1]$.
The constraint is not active in the region where $ \frac{d\mu_{\alpha}}{dx} < \frac{1}{2}$.
This is the non-saturated region and its closure is $\supp(\sigma-\mu_{\alpha}) = [-r,r]$.

\subsection{Two lemmas}

The connection between potential theory and the asymptotics theory of 
polynomials is well-known. If $P$ is a monic polynomial and 
\[ \nu = \frac{1}{n} \sum_{x : P(x) = 0} \delta_x \]
is its normalized zero counting measure (each zero is included in the sum according to 
its multiplicity), then 
\[ \frac{1}{n} \log |P(x)| = - U^{\nu}(x).  \]
If $(P_n)_n$ is a sequence of monic polynomials, and $(\nu_n)_n$
is the corresponding sequence of normalized zero counting measures then
the convergence of $(\nu_n)_n$ contains information on the $n$th root asymptotic behavior of the polynomials.
We need two such results.

\begin{lemma} \label{lemma21}
	
	\begin{enumerate} 
		\item[\rm (a)] Let $(P_n)_n$ be a sequence of monic polynomials,
	$\deg P_n \leq \alpha n$, having real and simple zeros, such that
	the zeros of $P_n$ are separated by the points of $E_n$ for every $n$. 
	Suppose that the sequence of normalized zero counting measures $(\nu_n)_n$
	where 	$\nu_n = \frac{1}{n} \sum_{x: P_n(x)=0} \delta_x$, 
	has a weak$^*$ limit $\mu$ as $n \to \infty$.	
	Then $\mu \leq \sigma$, and 
	\begin{equation}\label{minUmu1} \liminf_{n \to \infty} \frac{1}{n} \log \| P_n \|_{E_n}
	\geq - \min_{x \in \supp(\sigma-\mu)} U^{\mu}(x). 
	\end{equation}
	\item[\rm (b)] If $(P_n^*)_n$ is the sequence of extremal polynomials
	satisfying \eqref{Pnextremal} then $\nu_n \stackrel{\ast}{\to} \mu_{\alpha}$ as
	$ n \to \infty$, and equality holds
	\begin{equation} \label{minUmu2} 
		\lim_{n \to \infty} \frac{1}{n} \log \| P_n^* \|_{E_n}
	= - \min_{x \in \supp(\sigma-\mu)} U^{\mu_{\alpha}}(x). \end{equation}
	\end{enumerate}
\end{lemma}
\begin{proof}
	Part (a) is Lemma 4.2 of \cite{Rak96}, where it is stated
	under the assumption that the zeros are in $[-1,1]$.  
	See Lemma 5.5 in \cite{DS98} for the statement without
	this extra assumption.	
	Part (b) is in \cite[Theorem 2]{Rak96} or \cite[Theorem 3.3]{DS98}.
\end{proof}
Part (b) of Lemma \ref{lemma21} will be used in the proof of the lower bound,
while part (a) will be used in the proof of the upper bound, see the proof of Proposition \ref{prop32}.

\begin{remark} \label{remark22}
Note that the logarithmic potential $U^{\mu}$ of a positive measure $\mu$ 
is a lower semi-continuous function \cite{ST97}
and therefore its minimum over a compact (as in \eqref{minUmu1} and \eqref{minUmu2}, as well as in \eqref{minUmu3} below) exists.
 
In the situation of Lemma \ref{lemma21}, however,  the logarithmic potential $U^{\mu}$ 
is actually continuous. 
This follows from $\mu \leq \sigma$ and the fact that $U^{\sigma}$
is continuous. Indeed, $U^{\sigma-\mu}$ is lower semi-continuous,
and therefore $U^{\mu} = U^{\sigma} - U^{\sigma-\mu}$ is upper semi-continuous
as well, hence continuous. 
\end{remark}

The second lemma is probably well-known, but I was not able to find an
appropriate reference for it.
\begin{lemma} \label{lemma22}
	Suppose $(P_n)_n$ is a sequence of monic polynomials, $\deg P_n \leq \alpha n$,
	such that the zeros of all $P_n$ are in a compact set.  
		Suppose that the sequence of normalized zero counting measures $(\nu_n)_n$
	has a weak$^*$ limit $\mu$ as $n \to \infty$.	
	Then
	\begin{equation} \label{minUmu3} 
		\lim_{n \to \infty} \frac{1}{n} \log \|P_n\|_{[-1,1]} = - \min_{x \in [-1,1]} U^{\mu}(x).  
		\end{equation}
\end{lemma}

\begin{proof}
	By the principle of descent \cite{ST97} we have
	\[ U^{\mu}(x^*) \leq \liminf_{n \to \infty} U^{\nu_n}(x_n) \]
	whenever $x_n \to x^*$. 
	Since $|P_n(x)| = e^{-n U^{\nu_n}(x)}$, this means that  
	\[ \limsup_{n \to \infty} \frac{1}{n} \log |P_n(x_n)|  \leq - U^{\mu}(x^*) 
	\leq - \min_{x \in [-1,1]} U^{\mu}(x), \]
	whenever $(x_n)_n$ is a convergent sequence with a limit $x^* \in [-1,1]$.
	Taking $x_n \in [-1,1]$ with $|P_n(x_n)| = \|P_n\|_{[-1,1]}$
	and passing to convergent subsequences if necessary, we then find
	\begin{equation} \label{lemma21a} 
		\limsup_{n \to \infty} \frac{1}{n} \log \| P_n\|_{[-1,1]}  \leq - \min_{x\in [-1,1]} U^{\mu}(x).
		\end{equation}
	It remains to show that 	
	\begin{equation} \label{lemma21b} 
		\liminf_{n \to \infty} \frac{1}{n} \log \| P_n\|_{[-1,1]}  \geq - \min_{x\in [-1,1]} U^{\mu}(x).
	\end{equation}

	Suppose that \eqref{lemma21b} does not hold. Then there exist $\varepsilon > 0$
	and $x_0 \in [-1,1]$ such that
	\[ \frac{1}{n} \log \| P_n\|_{[-1,1]} \leq - U^{\mu}(x_0) - \varepsilon \]
	for an infinite number of $n$, say for $n \in \mathcal N \subset \mathbb N$.
	This means that
	\begin{equation} \label{lemma21c} 
		U^{\nu_n}(x) \geq U^{\mu}(x_0) + \varepsilon, \qquad x \in [-1,1], \ n \in \mathcal N. 
		\end{equation}
	By the lower envelope theorem \cite{ST97} and the assumption that $\nu_n \stackrel{\ast}{\to} \mu$ as $n \to \infty$,  we have
	\begin{equation} \label{lemma21d}
		U^{\mu}(x) = \liminf_{n \to \infty \atop
		n \in \mathcal N_1} U^{\nu_n}(x)  \quad \text{q.e.} \end{equation}
	where q.e.\ means quasi everywhere, i.e., the exceptional set is
	a polar set (a small set for potential theory).  
	In \eqref{lemma21c} we then take $n \to \infty$ along the subsequence $\mathcal N$
	and we find using \eqref{lemma21d}
	\[  
		U^{\mu}(x) \geq U^{\mu}(x_0) + \varepsilon  
		\quad \text{for q.e. } x \in [-1,1]. \]
	This means that  
	\begin{equation} \label{lemma21e} 
		\{x \in [-1,1] \mid U^{\mu}(x) < U^{\mu}(x_0) + \varepsilon \} 
		\end{equation}
	is a polar set, which is not the case, as we now show.
	
	If $U^{\mu}$ is a continuous function then \eqref{lemma21e} 
	contains an open non-empty interval and thus it is not a polar set, see e.g.\
	\cite[Example 5.2.7]{Helms09}. 	
	If $U^{\mu}$ is not continuous, then we come to the same
	conclusion, if we use certain more advanced
	results from potential theory, in particular around thinness and the fine
	topology, which we will not explain here, but see \cite{AG01, Helms09, ST97}  and other works
	on potential theory. The set 
	$\{ x \in \mathbb C \mid U^{\mu}(x) < U^{\mu}(x_0) + \varepsilon \}$
	is an open neighborhood of $x_0$ in the fine topology,
	and the interval $[-1,1]$ is not thin at $x_0 \in [-1,1]$, see \cite[Corollary 6.7.8]{Helms09}. 
	Therefore \eqref{lemma21e} is not a polar set.
	
	We thus arrive at a contradiction and we conclude that \eqref{lemma21b} does hold.
	Together with \eqref{lemma21a} we  get \eqref{lemma21} and the lemma follows.
\end{proof}

\subsection{Conclusion of the proof of the lower bound}

We apply the two lemmas to the
extremal polynomials $P_n^*$ satisfying \eqref{Pnextremal}. 
	By 	Lemma \ref{lemma21} (b)  we have
	\begin{equation} \label{PnonEn}  
		\lim_{n \to \infty} \frac{1}{n} \log \| P_n^*\|_{E_n} 
		= - \min_{x \in [-r,r]} U^{\mu_{\alpha}}(x), \end{equation}
	since $\supp(\sigma-\mu_{\alpha}) = [-r,r]$ by \eqref{mualpha}.
	From Lemma \ref{lemma22}  and \eqref{nuPnlimit} we get
	\begin{equation} \label{PnonI} 
		\lim_{n \to \infty} \frac{1}{n} \log \| P_n^* \|_{[-1,1]} 
		= - \min_{x \in [-1,1]} U^{\mu_{\alpha}}(x).
	\end{equation}
	Combining \eqref{PnonEn} and \eqref{PnonI} we obtain
	\begin{equation} \label{limPn} 
		\lim_{n \to \infty} \frac{1}{n} \log \frac{\| P_n^* \|_{[-1,1]}}{\|P_n^*\|_{E_n}}
		=  \min_{x \in [-r,r]} U^{\mu_{\alpha}}(x) - \min_{x \in [-1,1]}
		U^{\mu_{\alpha}}(x).
	\end{equation}

The limit \eqref{lowerbound2} and thereby the lower bound \eqref{lowerbound} 
follows from \eqref{limPn}  and the following proposition.
\begin{proposition} \label{prop24} 
We have
\begin{equation} \label{CalphaUmualpha}	
	C(\alpha) = 
	\min_{x \in [-r,r]} U^{\mu_{\alpha}}(x) - \min_{x \in [-1,1]}
	U^{\mu_{\alpha}}(x) \end{equation}
with $C(\alpha)$ as in \eqref{Calpha} above.
\end{proposition}

\begin{proof}
The derivative of $U^{\mu_{\alpha}}$ is a principal value integral
that can be calculated explicitly.  The result is
\begin{align} \nonumber  \frac{d}{dx} U^{\mu_{\alpha}}(x)
	& = - \fint \frac{d\mu_{\alpha}(y)}{x-y} \\
	& =  \frac{1}{2} \log\left(1-x^2\right) - \log\left(\alpha + \sqrt{x^2-r^2}\right),  \label{Identity1}
	\qquad \text{for } r < x < 1. 
\end{align}
We give the details of the calculations for \eqref{Identity1} later, after 
finishing the main line of the argument.

The derivative \eqref{Identity1} is negative for $r < x < 1$.
Therefore (and by symmetry) the minimum of $U^{\mu_{\alpha}}(x)$
over $[-1,-r] \cup[r,1]$ is at $x=  \pm 1$. Also $U^{\mu_{\alpha}}$ is constant
on $[-r,r]$ by \eqref{Umualpha}. Hence
\[ \min_{x \in [-r,r]} U^{\mu_{\alpha}}(x) - \min_{x \in [-1,1]} U^{\mu_{\alpha}}(x) = U^{\mu_{\alpha}}(r) - U^{\mu_{\alpha}}(1). \]
In view of \eqref{Identity1} and the fundamental theorem of calculus, we arrive at
\eqref{CalphaUmualpha} provided that  
\begin{equation} \label{Identity2} 
	C(\alpha) = \int_{r}^1 \left( 
	\log\left(\alpha + \sqrt{x^2-r^2} \right)
		- \frac{1}{2} \log\left(1-x^2\right) \right)  dx, 
	\end{equation}
with $r = r(\alpha) = \sqrt{1-\alpha^2}$.
Thus the proof of the proposition is complete up to the verification of the two 
identities \eqref{Identity1} and  \eqref{Identity2} to which we turn next.  

\paragraph{\it Proof of the identity \eqref{Identity1}.}
By \eqref{mualpha} and \eqref{mualpha2} the principal value integral in \eqref{Identity1} (with $x \in (r,1)$) splits into two parts
\begin{align} \nonumber
	\fint \frac{d\mu_{\alpha}(y)}{x-y} 
	& = \frac{1}{2} \fint_{-1}^1 \frac{1}{x-y} dy
	- \frac{1}{\pi} \int_{-r}^r \frac{1}{x-y} \arccos \left( \frac{\alpha}{\sqrt{1-y^2}}\right) dy \\
	& = \frac{1}{2} \log (1+x) - \frac{1}{2} \log(1-x)
	- I_{\alpha}(x) \label{Identity1a}
\end{align}
where
\[ I_{\alpha}(x) = \frac{1}{\pi} \int_{-r}^r \frac{1}{x-y} \arccos 	\left(\frac{\alpha}{\sqrt{1-y^2}}\right) dy \]
is a usual integral (not a principal value integral)
that converges for every $x > r$.
We integrate by parts
\[ I_{\alpha}(x) = -\frac{\alpha}{\pi} \int_{-r}^r \log(x-y) \frac{y}{1-y^2} \frac{1}{\sqrt{r^2-y^2}} dy, \quad x > r, \]
and then compute the derivative
\begin{align*}
	\frac{d}{dx} I_{\alpha}(x) & =
	- \frac{\alpha}{\pi} \int_{-r}^r \frac{1}{x-y} \frac{y}{1-y^2} \frac{1}{\sqrt{r^2-y^2}} dy \\
	& = - \frac{\alpha x}{1-x^2} \frac{1}{\sqrt{x^2-r^2}}
	- \frac{1}{2} \frac{1}{x-1} + \frac{1}{2} \frac{1}{x+1}
\end{align*}
by first turning the integral into an integral on a contour
around the interval $[-r,r]$ in the complex plane, and then
evaluating it by the residue theorem for the exterior domain. 
The result can be integrated again to give
\begin{equation} \label{Identity1b} I_{\alpha}(x) =
- \log\left(\alpha + \sqrt{x^2-r^2}\right) + \log (1+x),
	\qquad x > r,
\end{equation}
where we note that the constant of integration vanishes since $I_{\alpha}(x) \to 0$
as $x \to + \infty$.
Using this in \eqref{Identity1a} we obtain \eqref{Identity1}.

\paragraph{\it Proof of the identity \eqref{Identity2}.}

Observe that \eqref{Identity2} holds for $\alpha = 0$
since then both sides are equal to $0$.
Thus it is enough to show that the $\alpha$-derivatives of the two
sides agree.

For the left-hand side of \eqref{Identity2} we have by \eqref{Calpha}
\begin{equation} \label{dCalpha} 
	\frac{d}{d\alpha} C(\alpha) = \frac{\log(1+\alpha) - \log(1-\alpha)}{2}. \end{equation}
For the right-hand side we first compute the $\alpha$-derivative
of the integrand. Using $r= r(\alpha) = \sqrt{1-\alpha^2}$
we find by direct calculation
\[  \frac{d}{d \alpha} \left( 	
\log\left(\alpha + \sqrt{x^2-r^2} \right)
- \frac{1}{2} \log\left(1-x^2\right) \right)
	 = \frac{1}{\sqrt{x^2-r^2}}. \]
The integrand of \eqref{Identity2} vanishes
at $x = r$, and thus we obtain the following derivative of
the right hand side of \eqref{Identity2} 
\begin{align*} 
	\int_r^1 \frac{1}{\sqrt{x^2-r^2}} dx 
	& = \left. \log \left(x + \sqrt{x^2-r^2}\right) \right|_{x=r}^{x=1} \\
	& = \log\left(1+ \sqrt{1-r^2}\right) - \log r  \\
	& = \log \left(1+ \alpha\right) - \frac{1}{2} \log(1-\alpha^2) 	
\end{align*}
which after simplification agrees with \eqref{dCalpha}. 
\end{proof}

\section{Proof of the upper bound} \label{section3}

To prove the upper bound \eqref{upperbound} we start by
showing that the extremal polynomial $p_n^*$ has only real
zeros that are separated by the $n$-grid $E_n$.

\subsection{Zeros of the extremal polynomial}

We fix $0 < \alpha < 1$. For each $n$, we take a polynomial $p_n^*$
of degree $\leq \alpha n$ as in \eqref{pnstar} that we normalize 
such that 
\begin{equation} \label{pnstarnorm} 
	\| p_n^* \|_{[-1,1]} = 1 = p_n^*(x_n^*) \end{equation}
for some $x_n^* \in [-1,1]$.  It is clear that $x_n^* \not\in E_n$
since otherwise $\| p_n \|_{E_n} = 1$, and the polynomial
would not maximize the ratio.
\begin{lemma} \label{lemma31}
	\begin{enumerate}
		\item[\rm (a)] The polynomial $p_n^*$ minimizes
		$\| p\|_{E_n}$ among all polynomials $p$
		of degree $\leq \alpha n$ with $p(x_n^*) = 1$.
		\item[\rm (b)] The polynomial $q_n^*$ defined by
		\begin{equation} \label{qndef} 
			q_n^*(x) = x^{\lfloor \alpha n \rfloor}
			p_n^*\left(x_n^* + \tfrac{1}{x}\right) \end{equation}
		is a monic polynomial of degree $\lfloor \alpha n \rfloor$
		that minimizes the weighted uniform norm
		\begin{equation} \label{wnorm} 
			\max_{x \in \Sigma_n} \left| x^{-\lfloor \alpha n\rfloor} q(x)\right|,
		\end{equation}
		among all monic polynomials of degree $\lfloor \alpha n \rfloor$,
		where $\Sigma_n$ is the transformed grid,
		\begin{equation} \label{Sigman}
			\quad \Sigma_n = \{ (x-x_n^*)^{-1} \mid x \in E_n \}.
		\end{equation}
		\item[\rm (c)] 
		$p_n^*$ has only simple real zeros  and $\deg p_n^* \geq \lfloor \alpha n \rfloor -1$. 
		\item[\rm (d)] At least $\lfloor \alpha n \rfloor - 1$ zeros of $p_n^*$ are in $(-1,1)$.
		\item[\rm (e)] The zeros of $p_n^*$ are separated by the
		points in the $n$-grid $E_n$,
	\end{enumerate} 
\end{lemma}  

\begin{proof}
	(a) Suppose $p$ is a polynomial of degree $\alpha n$ with $p(x_n^*) = 1$ and $\|p\|_{E_n} < \| p_n^* \|_{E_n}$. 
	Since $x_n^* \in [-1,1]$ and $p(x_n^*) = 1$, we then have $\| p \|_{[-1,1]} \geq 1 = \| p_n^*\|_{[-1,1]}$, and therefore
	\[ \frac{\| p\|_{[-1,1]}}{\| p\|_{E_n}} > 
	\frac{\| p_n^* \|_{[-1,1]}}{\| p_n^*\|_{E_n}} \]
	which contradicts the extremal property \eqref{pnstar} of $p_n^*$.
	\medskip
	
	(b) It is easy to see that \eqref{qndef} is indeed a polynomial of
	degree $\lfloor \alpha n \rfloor$ and it is monic because $p_n^*(x_n^*) = 1$.
	
	Likewise, we associate to any polynomial $p$ of degree $\leq \alpha n$ 
	with $p(x_n^*) = 1$ the  monic polynomial 
	$q(x) = x^{\lfloor \alpha n\rfloor} p(x_n^*+ \frac{1}{x})$ of degree $\lfloor \alpha n \rfloor$.
	Then 
	\[ p(x) = (x-x_n^*)^{\lfloor \alpha n\rfloor} q\left(\tfrac{1}{x-x_n^*}\right) \]
	and, with $\Sigma_n$ as in \eqref{wnorm} 
	\begin{align*} \| p \|_{E_n} & = 
		\max_{x \in E_n} \left|(x-x_n^*)^{\lfloor \alpha n\rfloor} q\left(\tfrac{1}{x-x_n^*}\right) \right| = \max_{x \in \Sigma_n}
		\left| w(x) q(x) \right|
	\end{align*}
	with 
	\[ w(x) = |x|^{-\lfloor \alpha n \rfloor}. \]
	Because of part (a) we see that $q_n^*$ minimizes
	$\| w q \|_{\Sigma_n}$ among monic polynomials of degree
	$\lfloor \alpha n \rfloor$.
	
	\medskip
	(c) $q_n^*$ has only real zeros. Indeed if $z_0$
	were a non-real zero of $q_n^*$ then
	\[ q(x) = \frac{x-\Re z_0}{x-z_0} q_n^*(x) \]
	would be a monic polynomial of the same degree satisfying
	$|q(x)| < |q_n^*(x)|$ for every real $x$ that is not a zero of $q_n^*$.
	This would lead to a contradiction with part (b).
	
	Also the zeros of $q_n^*$ are simple, since if $x_0$ is a higher
	order real zero then for small enough $\varepsilon > 0$ the
	monic polynomial
	\[ q(x) =  \frac{(x-x_0-\varepsilon)(x-x_0+\varepsilon)  }{(x-x_0)^2}  q_n^*(x) \]
	would have a smaller weighted norm $\| wq\|_{\Sigma_n}$ than $q_n^*$.
	Because of \eqref{qndef} it then follows that $p_n^*$ has only 
	simple real zeros as well, since any zero $x_0 \neq 0$ of $q_n^*$ corresponds
	to the zero 
	\begin{equation} \label{zeroqtop} 
		x_n^* + \frac{1}{x_0} \end{equation} of $p_n^*$. 
	
	Since $q_n^*$ has $\lfloor \alpha n \rfloor$ simple real zeros, at
	least $\lfloor \alpha n \rfloor -1$ of them are different from $0$,
	and thus $p_n^*$ has at least that number of simple real
	zeros.
	In particular $p_n^*$ has degree $\geq \lfloor \alpha n \rfloor - 1$.
	
	\medskip
	(d) 
	The zeros of $q_n^*$ are separated by the points of $\Sigma_n$.
	Indeed if $x_1 < x_2$ are two zeros of $q_n^*$ and the
	interval $[x_1,x_2]$ would  not contain any points of $\Sigma_n$
	then 
	\[ x \mapsto  \frac{(x-x_1+\varepsilon)(x-x_2-\varepsilon)}{(x-x_1)(x-x_2)} q_n^*(x) \]
	would be a monic polynomial of the same degree with a strictly
	smaller weighted norm $\| w q_n \|_{\Sigma_n} <
		\| wq_n^* \|_{\Sigma_n}$, provided $\varepsilon > 0$
	is small enough, which would contradict part (b). 
	
	From \eqref{En} and \eqref{Sigman} it is easily checked that $(-1-x_n^*)^{-1}$ and $(1-x_n^*)^{-1}$ 
	are consecutive points in $\Sigma_n$. Thus at most one zero of $q_n^*$
	is in the closed interval between those points. This zero (if it exists) corresponds 
	by  \eqref{zeroqtop}  to a zero of $p_n^*$ in $(-\infty,-1] \cup [1,\infty)$, 
	or to $\infty$ in case $0$ is a zero of $q_n^*$. The other zeros of $q_n^*$
	correspond to zeros of $p_n^*$ in $(-1,1)$, and thus $p_n^*$
	has at least $\lfloor \alpha n \rfloor - 1$ zeros in $(-1,1)$.
	
	\medskip
	(e) 
	The zeros of $q_n^*$  belong to the open interval $(\min \Sigma_n, \max \Sigma_n)$. 
	Indeed, if $x_0$ is a zero of $q_n^*$ with $x_0 \geq \max \Sigma_n$,
	then $q(x) =  \frac{x-x_0+\varepsilon}{x-x_0}   q_n^*(x_0)$ has a smaller weighted
	norm $\| wq \|_{\Sigma_n}$ than $q_n^*$, if $\varepsilon > 0$ is small enough,
	and we get a similar contradiction if $x_0 \leq \min \Sigma_n$.
	
	We already proved that the zeros of $q_n^*$ are separated by the points in $\Sigma_n$,
	and so we conclude that the zeros of $q_n^*$ are also separated by the points in $\Sigma_n$
	when the real line is considered as part of the Riemann sphere. The separation
	is then preserved under the inversion that is included in the mapping \eqref{zeroqtop}
	to the zeros of $p_n^*$. This implies that the zeros of $p_n^*$ are
	separated by the gridpoints of $E_n$ as claimed in part (e). 
\end{proof}

\subsection{The functional $J$}

Recall from \eqref{Imualpha} that $\mu_{\alpha}$ minimizes $I(\mu)$ among measures
$\mu \in \mathcal M_{\alpha,\sigma}$. 
We consider another functional
\begin{equation} \label{Jmudef} 
	J(\mu) = \min_{x \in \supp(\sigma-\mu)} U^{\mu}(x)
	- \min_{x \in [-1,1]} U^{\mu}(x) \end{equation}
on measures $\mu \in \mathcal M_{\alpha,\sigma}$.
Note that by Proposition \ref{prop24} we have
\begin{equation} \label{Jmualpha} 
	J(\mu_{\alpha}) = C(\alpha) > 0, \end{equation}
since $\supp(\sigma-\mu_{\alpha}) = [-r,r]$.

\begin{proposition} \label{prop32}
	We have
	\[ \limsup_{n \to \infty} \frac{1}{n} \log \frac{\| p_n^* \|_{[-1,1]}}{\| p_n^*\|_{E_n}}
			\leq \sup_{\mu \in \mathcal M_{\alpha,\sigma}} J(\mu). 
		\]
\end{proposition}
\begin{proof}
We start by taking a subsequence $\mathcal N \subset \mathbb N$ such that
	\begin{equation} \label{pnstarratio1}
		\lim_{\mathcal N \ni n \to \infty}
	\frac{1}{n} \log \frac{\| p_n^*\|_{[-1,1]}}{\|p_n^*\|_{E_n}}
	= \limsup_{n \to \infty} 	\frac{1}{n} \log 
	\frac{\| p_n^* \|_{[-1,1]}}{\|p_n^* \|_{E_n}}. \end{equation}

The zeros of $p_n^*$ may not be uniformly bounded. However, by Lemma \ref{lemma31}(c), there is at most one zero outside $[-2,2]$. If there is such
a zero, say $x_0$, then we set
\begin{equation} \label{pnhat1} 
		\widehat{p}_n(x) =  \kappa_n^{-1} \frac{p_n^*(x)}{x-x_0}, 
		\end{equation}
where $\kappa_n$ is the leading coefficient of $p_n^*$. Otherwise we set
\begin{equation} \label{pnhat2} 
	\widehat{p}_n(x) = \kappa_n^{-1} p_n^*(x). \end{equation}
Then $\widehat{p}_n$ is a monic polynomial of degree 
$\lfloor \alpha n \rfloor$ or $\lfloor \alpha n \rfloor - 1$.
In case \eqref{pnhat2} we clearly have
\begin{equation} \label{pnhatratio1}
 \frac{ \| p_n^*\|_{[-1,1]}}{\| p_n^*\|_{E_n}}
		= 	\frac{ \| \widehat{p}_n \|_{[-1,1]}}{\| \widehat{p}_n\|_{E_n}},
 \end{equation}
 while in case \eqref{pnhat1} we can claim that
\begin{equation} \label{pnhatratio2}  \frac{1}{3} \frac{\| p_n^*\|_{[-1,1]}}{\|p_n^*\|_{E_n}} \leq \frac{\| \widehat{p}_n\|_{[-1,1]}}{\|\widehat{p}_n\|_{E_n}} \leq  3 \frac{\| p_n^*\|_{[-1,1]}}{\|p_n^*\|_{E_n}}. \end{equation}
To obtain \eqref{pnhatratio2} we note that since $|x_0| > 2$
we have $|x_0|-1 \leq |x-x_0| \leq |x_0|+1$ for $x \in [-1,1]$,
so that
\[ |\kappa_n^{-1}| \frac{|p_n^*(x)|}{|x_0|+1} \leq   | \widehat{p}_n(x)| \leq 	|\kappa_n^{-1}| \frac{|p_n^*(x)|}{|x_0|-1}, \qquad \text{for } x \in [-1,1]. \]
Taking the supremum over $x \in [-1,1]$ and over $x\in E_n$, we obtain 
\begin{align*} 
		\frac{|\kappa_n^{-1}|}{|x_0|+1} \|p_n^*\|_{[-1,1]} \leq   \| \widehat{p}_n\|_{[-1,1]} \leq  \frac{|\kappa_n^{-1}|}{|x_0|-1} \|p_n^*\|_{[-1,1]},  \\
		\frac{|\kappa_n^{-1}|}{|x_0|+1}   
		\|p_n^*\|_{E_n} \leq   \| \widehat{p}_n\|_{E_n} \leq 	
			\frac{|\kappa_n^{-1}|}{|x_0|-1} \|p_n^*\|_{E_n}.  \end{align*}
Taking ratios of these inequalities leads to \eqref{pnhatratio2} since $\frac{|x_0|+1}{|x_0|-1} <3$ for $|x_0| > 2$.

From \eqref{pnstarratio1} \eqref{pnhatratio1}, \eqref{pnhatratio2} it then follows that
	\begin{equation} \label{pnstarratio2}
 \limsup_{n \to \infty} \frac{1}{n} \frac{\|p_n^*\|_{[-1,1]}}{\|p_n^*\|_{E_n}}
 	= \lim_{\mathcal N \ni n \to \infty}
	\frac{1}{n} \log \frac{\| \widehat{p}_n\|_{[-1,1]}}{\|\widehat{p}_n\|_{E_n}}.
	\end{equation}

Next,  by taking a further subsequence if necessary, we may also
assume that the sequence $(\nu_n)_n$ of normalized zero counting measures, i.e.,
\[ \nu_n = \frac{1}{n} \sum_{x : \widehat{p}_n(x) = 0} \delta_x \]
converges in the weak$^*$ sense as $n \to \infty$ with $n \in \mathcal N$.
Here we use Helly's selection theorem, and Lemma \ref{lemma31} (c).
By Lemma \ref{lemma31} (d) the weak$^*$ limit, 	say $\mu$, belongs to
$\mathcal M_{\alpha, \sigma}$. 

Now we apply Lemmas \ref{lemma21} and \ref{lemma22} to the polynomials $\widehat{p}_n$.
From part (a) of Lemma~\ref{lemma21}  we get
\[ \liminf_{\mathcal N \ni n \to \infty}
	\frac{1}{n} \log \| \widehat{p}_n\|_{E_n} \geq
	- \min_{x \in \supp(\sigma-\mu)} U^{\mu}(x) \]
and from Lemma \ref{lemma22}
\[ \lim_{\mathcal N \ni n \to \infty}
\frac{1}{n} \log \| \widehat{p}_n\|_{[-1,1]} =
- \min_{x \in [-1,1]} U^{\mu}(x). \]
These limits and \eqref{pnstarratio2} then imply that
\[ \limsup_{ n \to \infty}
\frac{1}{n} \log \frac{\| p_n^*\|_{[-1,1]}}{\|p_n^*\|_{E_n}}
	\leq J(\mu) \]
and the proposition since $\mu \in \mathcal M_{\alpha,\sigma}$.
\end{proof}

\subsection{Conclusion of the proof of the upper bound}

In view of Proposition \ref{prop32} it remains to show
that $\sup\limits_{\mu \in \mathcal M_{\alpha,\sigma}} J(\mu) = C(\alpha)$
in order to obtain \eqref{upperbound}. This is the final result of the paper.  

\begin{proposition} \label{prop33} For every $\mu \in \mathcal M_{\alpha,\sigma}$
	with $\mu \neq \mu_{\alpha}$ we have
\begin{equation} \label{Jextremal} 
	J(\mu)  < J(\mu_{\alpha}) = C(\alpha). 
\end{equation}
\end{proposition}
\begin{proof} We already noted in \eqref{Jmualpha} that $J(\mu_{\alpha}) = C(\alpha)  > 0$.
	
Let $\mu \in \mathcal M_{\alpha,\sigma}$ with $\mu \neq \mu_{\alpha}$.
Take $x_0 \in [-1,1]$ with 
\begin{equation} \label{minprinc1} U^{\mu}(x_0) = \min_{x\in [-1,1]} U^{\mu}(x). 
	\end{equation}
If $x_0 \in \supp(\sigma-\mu)$, then the minimimu of  $U^{\mu}$ over
$\supp(\sigma-\mu)$ is also attained at $x_0$, and it would follow from \eqref{Jmudef} that $J(\mu) = 0$.
Then the strict inequality \eqref{Jextremal} holds.
Hence we may assume that $x_0 \in [-1,1] \setminus \supp(\sigma-\mu)$.

Since $\mu \leq \sigma$ and $\mu_{\alpha} \leq \sigma$ we see that
both $U^{\mu}$ and $U^{\mu_{\alpha}}$ are continuous functions on $\mathbb C$, 
see also Remark \ref{remark22}, and they are 
both harmonic in $\mathbb C \setminus [-1,1]$.
Also $\mu \geq  \mu_{\alpha}$ on $[-1,1] \setminus \supp(\sigma-\mu)$, 
and therefore 
$U^{\mu - \mu_{\alpha}}$ is superharmonic on $\mathbb C \setminus \supp(\sigma-\mu)$.
It has a finite limit at infinity since $\mu$ and $\mu_{\alpha}$ have the same total mass,
and therefore $U^{\mu-\mu_{\alpha}}$, when viewed as a function on the Riemann sphere, 
extends to a function that is harmonic at infinity,  cf.\ \cite[Corollary 5.2.3]{AG01}.
The minimum principle for superharmonic functions \cite{Helms09,ST97} 
then tells us that the minimum of $U^{\mu-\mu_{\alpha}}$ is taken
on $\supp(\sigma-\mu)$ only. In particular, since $x_0 \not\in \supp(\sigma-\mu)$
\begin{equation} \label{minprinc} 
	U^{\mu - \mu_{\alpha}}(x_0) >  
	\min_{x \in \supp(\sigma-\mu)} U^{\mu-\mu_{\alpha}}(x). 
	\end{equation}
Combining \eqref{minprinc} with the obvious inequality (since $\supp(\sigma-\mu) \subset [-1,1]$)
\begin{align*}
	\min_{x \in \supp(\sigma-\mu)} U^{\mu - \mu_{\alpha}}(x)
	& \geq \min_{x \in \supp(\sigma-\mu)} U^{\mu}(x)
	- \max_{x \in [-1,1]} U^{\mu_{\alpha}}(x),
\end{align*}
we obtain 
\begin{align*} U^{\mu}(x_0) - U^{\mu_{\alpha}}(x_0) >
	\min_{x \in \supp(\sigma-\mu)} U^{\mu}(x) -
		\max_{x \in [-1,1]} U^{\mu_{\alpha}}(x), \end{align*}
which leads to
\begin{align} \nonumber 
	\min_{x \in \supp(\sigma-\mu)} U^{\mu}(x) - U^{\mu}(x_0)
	& < \max_{x \in [-1,1]} U^{\mu_{\alpha}}(x) -  U^{\mu_{\alpha}}(x_0) \\
	& \leq \max_{x \in [-1,1]} U^{\mu_{\alpha}}(x) - 
		\min_{x \in [-1,1]} U^{\mu_{\alpha}}(x). \label{minprinc2}
	  \end{align}
The left-hand side of \eqref{minprinc2} is equal to $J(\mu)$ because
of \eqref{Jmudef} and \eqref{minprinc1}.
For the right-hand side, we note that by the special property \eqref{Umualpha} 
of $U^{\mu_{\alpha}}$   we have
  \[ \max_{x \in [-1,1]} U^{\mu_{\alpha}}(x) = \ell_{\alpha}
  	= \min_{x \in \supp(\sigma-\mu_{\alpha})} U^{\mu_{\alpha}}(x), \]
and therefore the right-hand side of \eqref{minprinc2} is equal to $J(\mu_{\alpha})$. Thus $J(\mu) < J(\mu_{\alpha})$ and the proposition is proved.
\end{proof}

\begin{remark}
According to Proposition \ref{prop33} the constrained equilibrium measure
$\mu_{\alpha}$ is the unique maximizer of $J(\mu)$ among measures $\mu \in \mathcal M_{\alpha,\sigma}$. 
We may conclude from this that the sequence of normalized zero counting
measures of the extremal polynomials $p_n^*$ tends to the constrained
equilibrium measure $\mu_{\alpha}$ as $n \to \infty$. 
This follows from  the proof of Proposition \ref{prop32}, combined with the proven 
fact that
\[ \lim_{n\to \infty}
	\frac{1}{n} \log \frac{\| p_n^*\|_{[-1,1]}}{\|p_n^*\|_{E_n}} = C(\alpha), \]
as this gives that the weak$^*$ limit of any convergent subsequence is a measure
$\mu \in \mathcal M_{\alpha, \sigma}$ with $J(\mu) = C(\alpha)$.
Because of \eqref{Jextremal} this limit has to be $\mu_{\alpha}$, and thus
by a compactness argument the full sequence tends to $\mu_{\alpha}$ indeed.
\end{remark}

\subsection*{Acknowledgement} I want to thank Daan Huybrechs and Nick Trefethen for
their interest in this work, for useful discussions,  and for 
stimulating me to write the details of the proof of Theorem \ref{theorem11}. 

The author  was
supported by the long term structural funding "Methusalem grant of the Flemish Government"
and by FWO Flanders projects EOS 30889451 and G.0910.20.

\Addresses

\end{document}